\documentclass[11pt]{article}
\usepackage{mathrsfs}
\usepackage{mathtools}
\usepackage{amssymb}
\usepackage{amsmath}
\usepackage{amsthm}
\usepackage{amscd}
\usepackage{commath}
\usepackage{enumerate}
\usepackage{stmaryrd}
\usepackage{enumitem}
\usepackage{hyperref}
\usepackage{thmtools}
\usepackage{listings}
\usepackage{tikz}
\usepackage{tikz-cd}
\usepackage{comment}
\usepackage{lipsum}
\usetikzlibrary{matrix,arrows,decorations.pathmorphing}
\usepackage{bm}

\newtheorem{theorem}{Theorem}[section]
\newtheorem{lemma}[theorem]{Lemma}
\newtheorem{proposition}[theorem]{Proposition}

\newtheorem*{theorem*}{Theorem}
\newtheorem*{proposition*}{Proposition}

\theoremstyle{definition}

\theoremstyle{remark}
\newtheorem{remark}[theorem]{Remark}

\newcommand{\Q}{\mathbb{Q}}
\newcommand{\Z}{\mathbb{Z}}

\newcommand{\R}{\mathbb{R}}
\newcommand{\C}{\mathbb{C}}

\newcommand{\OCp}{\mathcal{O}_{\mathbb{C}_p}}
\newcommand{\Spec}{\mathrm{Spec}}
\newcommand{\Sp}{\mathrm{Sp}}
\newcommand{\an}{\mathrm{an}}

\newcommand{\mdr}{M_{dR}}
\newcommand{\mb}{M_B}
\newcommand{\mdol}{M_{Dol}}
\newcommand{\Pic}{\mathrm{Pic}}
\newcommand{\Picu}{\mathrm{Pic}^\natural_{A^\vee}(\C)}
\newcommand{\Gm}{\mathbb{G}_m}
\newcommand{\Gmcp}{\mathbb{G}_{m,\C_p}}
\newcommand{\rh}{\mathrm{RH}}
\newcommand{\topo}{\mathrm{top}}

\renewcommand{\Re}{\operatorname{Re}}

\DeclareMathOperator{\Hom}{Hom}

\DeclareMathOperator{\Lie}{Lie}
\DeclareMathOperator{\pgl}{PGL}
\DeclareMathOperator{\gl}{GL}

 \newenvironment{itemize*}
  {\begin{itemize}[topsep=-\parskip+\jot,itemsep=-\parskip-\jot]}
  {\end{itemize}}

\newenvironment{enumerate*}
  {\begin{enumerate}[label=(\alph*),topsep=-\parskip+\jot,itemsep=-\parskip-\jot]}
  {\end{enumerate}}

\newenvironment{enumerate**}
  {\begin{enumerate}[label=(\roman*),topsep=-\parskip+\jot,itemsep=-\parskip-\jot]}
  {\end{enumerate}}

\title{Bi-algebraicity in the rank one Riemann-Hilbert correspondence via o-minimality}
\date{\today}
\author{Abhishek Oswal}

\begin{document}
\setlist[description]{font=\normalfont\itshape\textbullet\space}
\maketitle

    
\section{Introduction}\label{introduction}
Let $A$ be a smooth, projective, irreducible, complex algebraic variety. Given a rank $r$ algebraic vector bundle $\mathcal{E}$ with integrable connection $\nabla : \mathcal{E} \rightarrow \mathcal{E} \otimes_{\mathcal{O}_A} \Omega^1_A$ on $A$, the sheaf of flat sections of $\nabla^\an$, $\ker(\nabla^\an)$, is a $\C $-local system of rank $r$ on $A^\an$. The association $(\mathcal{E},\nabla)\mapsto \ker(\nabla^\an)$ provides a rank-preserving, tensor functor from the category of algebraic vector bundles on $A$ equipped with an integrable connection to the category of $\C$-local systems on $A^\an$. This functor is in fact an equivalence of categories, classically known as the Riemann--Hilbert correspondence.

For this note, we shall be primarily concerned with \emph{rank one} local systems. We let $\mdr(A)$ (the rank one de Rham moduli space attached to $A$) denote the moduli space of line bundles on $A$ equipped with an integrable connection and let $\mb(A)$ (called the rank one Betti moduli space attached to $A$) denote the moduli space of rank one $\C$-local systems on $A^\an$.
More precisely, for a test $\C$-scheme $T$, $\mdr(A)(T)$ is the set of isomorphism classes of line bundles $\mathcal{L}$ on $A \times_{\operatorname{Spec}(\C)} T =: A_T$ equipped with a flat connection $\nabla : \mathcal{L}\rightarrow \mathcal{L}\otimes_{\mathcal{O}_{A_T}} \Omega^1_{A_T/T}$ relative to $T$ and $\mb(A)(T)$ is the set $\Hom(\pi_1(A^\an,a),\mathcal{O}_T(T)^\times)$ for a base point $a \in A(\C)$ (however we recall that $\mb(A)$ is nonetheless independent of the chosen base point). We remark that when the rank is greater than one, these moduli spaces are no longer fine moduli spaces. We refer the reader to the papers \cite{simpson1994moduliI} and \cite{simpson1994moduliII} where the construction of such moduli spaces in general rank has been carried out in detail. 

In the rank one case, both $\mb(A)$ and $\mdr(A)$ are naturally algebraic groups over $\C$. On $\C$-points the group operation is that of tensor product of local systems in the case of $\mb$ and that of tensor products of bundles equipped with the natural tensor product connection in the case of $\mdr$. 
The Riemann--Hilbert correspondence provides an isomorphism of complex analytic groups $\rh : \mdr(A)^\an \xrightarrow{\sim} \mb(A)^\an$. However, this isomorphism is almost never algebraic. 
To see this, one notes that the identity component of $\mdr(A)$ is the universal vector extension of $\Pic^0(A)$, which is an extension of $\Pic^0(A)$ by the vector group associated to the space of global one-forms $H^0(A,\Omega^1_A)$ on $A$.
On the other hand, the identity component of the Betti moduli space $\mb(A)$ is an affine algebraic torus. 
It is easy to see that on the vector space of one-forms $H^0(A,\Omega^1_A)$, the Riemann--Hilbert correspondence $\rh$ is an exponential map, and therefore cannot be algebraic (except in degenerate situations).

Let us call an algebraic subvariety of either $\mdr(A)$ or $\mb(A)$ \emph{bi-algebraic}, if its image (resp. pre-image) under $\rh$ is also algebraic. In \cite{simpson1993subspaces}, Simpson proves the following remarkable theorem that characterises the subvarieties of the rank one Betti and de Rham moduli spaces that are bi-algebraic with respect to the Riemann--Hilbert isomorphism. 

\begin{theorem}[Simpson - \cite{simpson1993subspaces}]\label{o-minimal-theorem}
    Let $S'$ be an irreducible algebraic subvariety of $\mdr(A)$ such that $S := \rh(S') \subseteq \mb(A)$ is also complex algebraic. Then $S$ must be a translate of a subtorus of $\mb(A)$.     
\end{theorem}

More precisely, Simpson proves that for an irreducible bi-algebraic subvariety $S \subseteq \mb(A)$, there is a map from $A$ to an \emph{abelian variety} $A'$ such that $S$ is the image of the homomorphism $\mb(A')\rightarrow \mb(A)$ induced by the pullback of rank one local systems. In this note, we provide a new proof of this theorem of Simpson, using tools from the theory of o-minimal structures. This is carried out in \S \ref{o-minimal-proof}.

The second goal of this note is to adapt the o-minimal proof to a $p$-adic setting, namely that of Mumford curves over non-archimedean fields. Let $(\C_p,|\cdot|)$ denote the completion of an algebraic closure $\overline{\Q}_p$ of $\Q_p$, with respect to the unique extension to $\overline{\Q}_p$  of the $p$-adic absolute value of $\Q_p$. 
Consider a Schottky group $\Gamma$ inside $\pgl_2(\C_p)$; that is to say that $\Gamma$ is a finitely-generated, discontinuous subgroup of $\pgl_2(\C_p)$ with no non-trivial elements of finite order (see \cite{gerritzen1980schottky} for the precise definitions). Such a group is automatically a free group, say of rank $g$. 
Associated to $\Gamma$ is the analytic subdomain $\Omega_\Gamma \subseteq \mathbb{P}^{1,\an}_{\C_p}$ obtained as the complement of the limit set $\mathcal{L}_\Gamma \subseteq \mathbb{P}^{1,\an}_{\C_p}$ so that $\Gamma$ acts properly and discontinuously on $\Omega_\Gamma$.  The $\C_p$-analytic space obtained as the quotient $\Omega_\Gamma/\Gamma$ is in fact a smooth, proper $\C_p$-analytic space, and thus naturally the analytic space associated to a genus $g$ smooth, projective algebraic curve $X$ over $\C_p$. Curves obtained in this manner are called Mumford curves. If we view the analytic spaces as Berkovich spaces, the analytic uniformizing map $p : \Omega_\Gamma \rightarrow X^\an$ is in fact a topological universal cover of $X^\an$ and $\pi_1^{\topo}(X^\an)$ is isomorphic to $\Gamma$. 

One may attach to any representation $\lambda : \Gamma \rightarrow \gl_n(\C_p)$ of the Schottky group, a vector bundle of degree $0$ and rank $n$ on the Mumford curve $X$ in the following manner. Given a representation $\lambda : \Gamma \rightarrow \gl_n(\C_p)$, consider the trivial rank $n$ bundle with trivial connection on the universal cover, $\Omega_\Gamma \times \C_p^n \rightarrow \Omega_\Gamma$. 
We have an action of $\Gamma$ on $\Omega_\Gamma\times \C_p^n$ wherein $\Gamma$ acts via the representation $\lambda$ on the second factor. 
The quotient $(\Omega_\Gamma\times \C_p^n)/ \Gamma \rightarrow \Omega_\Gamma/\Gamma = X^\an$ defines a rank $n$ analytic vector bundle on $X^\an$ with a natural connection which by GAGA is uniquely algebraizable. 
Just as in the complex setting one may again consider this association on the level of moduli spaces. For the purposes of this note we shall only be concerned with the rank one moduli spaces. As such, we define $\mb(X)$ as the affine algebraic torus $\Spec(\C_p[\Gamma^\mathrm{ab}])$ over $\C_p$. As before, $\mdr(X)$ denotes the moduli space of algebraic line bundles on $X$ equipped with a connection. As in the complex case, $\mdr(X)$ has dimension $2g$, however this time the Betti space $\mb(X)$ has dimension $g$, half of that of the de Rham moduli space. The association of a line bundle with connection to every representation $\Gamma \rightarrow  \C_p^\times$ sketched above, gives rise to an injective \emph{rigid-analytic} morphism, $\rho : \mb(X)^\an \hookrightarrow \mdr(X)^\an$. The fact that this map is rigid-analytic appears to be well-known to experts (see for instance \cite[p.\,143]{andre1998p} and \cite[\S 1.5.6]{andre2003period}).

Although in this $p$-adic setting of Mumford curves, the Betti space is half the dimension of the de Rham space, it is a natural question to ask when an algebraic subvariety of $\mdr(X)^\an$ intersects $\mb(X)^\an$ in an algebraic subvariety of $\mb(X)$. In \S \ref{p-adic-proof}, we prove the following $p$-adic analogue of Simpson's result:

\begin{theorem}\label{p-adic-theorem}
    Let $X$ be a Mumford curve over $\C_p$ of genus $g \geq 1$. Let $\mb(X)$ denote the moduli space of rank one topological $\C_p$-local systems on $X^\an$ and $\mdr(X)$ the moduli space of (algebraic) line bundles with integrable connection on $X$. Consider the rigid analytic immersion $\mb(X)^\an \hookrightarrow \mdr(X)^\an$ as above. Suppose $S' \subseteq \mdr^\an$ is an algebraic subvariety such that $S := S' \cap \mb^\an$ is an irreducible algebraic subvariety of $\mb$. Then $S$ is a translate of a subtorus of $\mb$. 
\end{theorem}

\subsection*{Strategy of the proof}
The idea to use o-minimality is inspired by the several recent works wherein o-minimality has played a fundamental role in establishing key `functional transcendence' theorems (see for instance \cite{bakker2019ax} and \cite{mok2019ax}).

In particular, two algebraization theorems from o-minimal geometry, have been central to the proofs of these transcendence results. The first of these theorems is the counting theorem of Pila--Wilkie \cite{pila2006rational}, and the second is the definable Chow theorem of Peterzil--Starchenko \cite{starchenko2008nonstandard} which states that any closed, complex-analytic subset of a complex algebraic variety that is also definable in some o-minimal structure, is an algebraic subset.   

In the $p$-adic setting, the role of o-minimal structures in our proofs is replaced by the theory of rigid subanalytic sets developed by Lipshitz--Robinson (\cite{lipshitz1993rigid}, \cite{lipshitz2000rings}). The theory of rigid subanalytic sets provides a class of subsets of $\C_p^n$ that have `tame' topological properties similar to those satisfied by subsets of $\R^n$ definable in o-minimal structures.

To explain the main idea behind our proof, let us restrict to the complex setting. By replacing the smooth, projective, irreducible, complex algebraic variety $A$ by its Albanese variety, we may even assume that $A$ is itself a complex abelian variety, say of dimension $g$. In this case, the Betti moduli space $\mb(A)(\C)$ is isomorphic to the affine algebraic torus $(\C^\times)^{2g}$. 

Let us endow the Betti and de Rham moduli spaces with their natural $\R^{\mathrm{an},\exp}$-definable structures (that extend their complex algebraic structures). The fact that the analytic isomorphism $\rh$ is transcendental is in particular reflected in the fact that the analytic isomorphism $\rh$ is not definable in any o-minimal structure (by the definable Chow theorem). Thus, the isomorphism $\rh : \mdr(A)(\C)\simeq (\C^\times)^{2g}$ induces on the affine algebraic torus $(\C^\times)^{2g}$ two distinct $\R^{\mathrm{an}, \exp}$-definable structures. Consider the universal covering map $E : \C^{2g} \rightarrow (\C^\times)^{2g}$ that sends $(z_1,\ldots,z_{2g})\mapsto (\exp(2\pi i z_1),\ldots, \exp(2 \pi i z_{2g})).$ The two different definable structures on $(\C^\times)^{2g}$ may be interpreted as choosing two distinct fundamental domains for the action of $\Z^{2g}$ (acting via deck transformations) on the universal covering map $E : \C^{2g} \rightarrow (\C^\times)^{2g}$. It is easy to see that the restriction of $E$ to the fundamental domain $F := \{(z_1,\ldots,z_{2g}) \in (\C^\times)^{2g} : 0 \leq \operatorname{Re}(z_i) < 1 \}$ is an $\R^{\mathrm{an},\exp}$-definable bijection onto the Betti definable structure on $(\C^\times)^{2g}$. On the other hand, the definable structure on $(\C^\times)^{2g}$ induced from the de Rham side, corresponds to the choice of a skew fundamental domain in $\C^{2g},$ that we shall call $G$. It turns out that the intersection $F\cap G$, of these two fundamental domains is a \emph{bounded} subset of $\C^{2g}$ (see \autoref{bounded-lemma}). 

Now, a subvariety of $(\C^\times)^{2g}$ that is bi-algebraic for the Riemann--Hilbert map is thus an analytic subvariety that is definable with respect to both these structures. One may now imagine that there is a certain obstruction for an analytic subset to be definable in both these structures. Indeed, for a definable subset $S$ of $F$, its corresponding set in $G$, is obtained by translating $S \cap (G- \vec{n})$ by $\vec{n}$ for every $\vec{n} \in \Z^{2g}$. Since $F \cap G$ is bounded, this means that if $S$ is unbounded, then one would have to chop $S$ into infinitey many pieces each of which would have to be translated by a different vector $\vec{n} \in \Z^{2g}$. This makes it difficult for the image of $S$ in $G$ to have finitely many connected components for instance, in particular, for it to be definable.

\subsection*{The Dolbeault moduli space and the non-abelian Hodge correspondence}

For a smooth, projective complex algebraic variety $A$, in addition to the Betti and de Rham moduli spaces, there is a third closely related moduli space, $\mdol(A,r)$, that is the moduli space of rank $r$ Higgs bundles on $A$ (also called the rank $r$ Dolbeault moduli space attached to $A$).     

In \cite{simpson1994moduliI} and \cite{simpson1994moduliII}, Simpson establishes a \emph{real-analytic} isomorphism $\mdol(A,r)(\C)\xrightarrow{\sim} \mb(A,r)(\C)$ often referred to as the non-abelian Hodge correspondence. 

Moreover, in the rank one case, we recall that the identity component $\mdol(A)(\C)^\circ$ of the Dolbeault moduli space is isomorphic to the product $\Pic^\circ(A)(\C) \times H^0(A,\Omega^1_A)$ and the identity component $\mb(A)(\C)^\circ$ is isomorphic to the affine torus $(\C^\times)^{2g}$ where $g = \dim(\Pic^\circ(A)) = \dim(H^0(A,\Omega^1_A))$. The rank one non-abelian Hodge correspondence $\Pic^\circ(A)(\C) \times H^0(A,\Omega^1_A)\simeq (\C^\times)^{2g} $ is the map that identifies $\Pic^\circ(A)(\C) \xrightarrow{\sim} (S^1)^{2g}$ and $H^0(A,\Omega^1_A) \xrightarrow{\sim} \R_{>0}^{2g}$ via \emph{real exponentials} (see \cite[Example on p.\,21]{simpson1992higgs}). Thus, the non-abelian Hodge correpondence $\mdol(A)(\C) \xrightarrow{\sim} \mb(A)(\C)$ in rank one, is an $\R^{\mathrm{an},\exp}$-definable real analytic isomorphism. 

In particular, if we have an algebraic subvariety $S_{Dol} \subseteq \mdol(A)$ such that its corresponding subset $S_{dR} \subseteq \mdr(A)$ is also algebraic in $\mdr$, then on the one hand $S_B \subseteq \mb(A)$ is $\R^{\mathrm{an},\exp}$-definable since $\mdol(A)(\C)\xrightarrow{\sim}\mb(A)(\C)$ is definable, and on the other hand $S_B \subseteq \mb(A)$ is complex analytic since $\rh : \mdr(A)(\C) \xrightarrow{\sim} \mb(A)(\C)$ is complex-analytic. Thus, by the definable Chow theorem, $S_B$, being both definable and complex analytic is \emph{algebraic} and therefore by \autoref{o-minimal-theorem}, $S_B$ is a finite union of translates of subtori.

This raises the question of whether the non-abelian Hodge correspondence is also o-minimally definable in arbitrary rank, which we think would have interesting consequences. For instance, an immediate consequence would be that a subvariety that is bi-algebraic in both $\mdol$ and $\mdr$, must automatically be tri-algebraic (i.e. algebraic in $\mb$ as well).


\subsection*{Acknowledgements} It is a great pleasure to thank Peter Sarnak for pointing me to Simpson's work, asking me the prescient question of whether o-minimal techniques could help prove such results and also for encouraging conversations. I am very grateful to Jacob Tsimerman for several helpful ideas and suggestions and to Ananth Shankar, Carlos Simpson, and Akshay Venkatesh for stimulating discussions. 

\section{An o-minimal proof of Simpson's result}\label{o-minimal-proof}

We now turn to the proof of \autoref{o-minimal-theorem}. We recall the setup of the theorem. $A$ is a smooth, projective, connected complex algebraic variety. The Albanese map of $A$ into its Albanese variety $\mathrm{Alb}(A)$ induces via pullback of local systems (resp. line bundles with flat connection) a finite \emph{algebraic} map of the corresponding (rank one) moduli spaces $\mb(\mathrm{Alb}(A)) \rightarrow \mb(A)$ (resp. $\mdr(\mathrm{Alb}(A)) \rightarrow \mdr(A)$), with these maps being compatible with the Riemann-Hilbert isomorphisms $\rh : \mdr^\an \xrightarrow{\sim} \mb^\an$. We may thus replace $A$ by its Albanese, and assume from now on that $A$ is a complex abelian variety, say of dimension $g$.

The Betti moduli space is an affine torus over $\C$ of dimension $2g$. Indeed, $\mb(A)(\C) = \Hom(\pi_1(A,0),\C^*) \approx (\C^*)^{2g}$, where the last isomorphism rests on a choice of basis of $H_1(A(\C),\Z).$ The de Rham moduli space $\mdr(A)$ is the universal vector extension $\Pic^\natural_{A^\vee}(\C)$ of the dual abelian variety $A^\vee$ that fits into the exact sequence:
\[0 \rightarrow \Gamma(A,\Omega^1_{A/\C}) \rightarrow \Picu \xrightarrow{\pi} A^\vee(\C) \rightarrow 0.\] On $\C$-points, $\pi$ is the map that forgets the connection; that is, $\pi$ takes a line bundle with flat connection, $(\mathcal{L},\nabla)$ to the line bundle $\mathcal{L}$.
We shall prove the following slightly more general claim:

\begin{proposition}\label{bi-algebraicity-prop}
    Suppose $P$ is an algebraic group over $\C$ that is an extension of a complex abelian variety $B$ of dimension $g$ by a vector group $\Omega$ over $\C$ of dimension $g$. Thus there is an exact sequence of complex algebraic groups $0 \rightarrow \Omega \rightarrow P \xrightarrow{\pi} B \rightarrow 0.$ Suppose $T$ is a complex affine torus of dimension $2g$; thus $T \approx \mathbb{G}_{m,\C}^{2g}$. 
    Suppose we have a complex analytic group isomorphism $\rho : T^\an \xrightarrow{\sim} P^\an$ such that $T(\R) \cap \rho^{-1}(\Omega)$ is a discrete subset of $T(\R)$ (in the Euclidean metric topology). \footnote{Note that for any \emph{algebraic} isomorphism $\phi : \mathbb{G}_{m,\C}^{2g} \xrightarrow{\sim} T$, the image $\phi(\mathbb{G}_m^{2g}(\R)) \subseteq T(\C)$ is independent of $\phi$, since all algebraic group automorphisms of $\mathbb{G}_{m,\C}^{2g}$ are defined over $\Q$. We call $T(\R)$ the image $\phi(\mathbb{G}_m^{2g}(\R))$ for any such $\phi$ as above and we may view $T(\R)$ as a real Lie subgroup of $T(\C)$.\label{def-of-T(R)}}  
    Then if $S \subseteq T$ is an irreducible algebraic subvariety such that $S' := \rho(S)$ is also algebraic in $P$, then $S$ is a translate of a subtorus of $T$. 
\end{proposition} 

To see how \autoref{bi-algebraicity-prop} applies to yield a proof of \autoref{o-minimal-theorem}, we only need to check that for a complex abelian variety $A$ of dimension $g$, $\rh^{-1}(\mb(\R))\cap \Gamma(A, \Omega^1_{A/\C})$ is discrete in $\mdr(\C).$ Since this discreteness property is invariant under finite quotients, we may as well assume that the complex abelian variety is principally polarized. In this case, we may choose algebraic coordinates on $\Gamma(A, \Omega^1_{A/\C}) \simeq \C^g$ and on $\mb(\C) \simeq (\C^*)^{2g}$ so that the map $\rh\vert_{\Gamma(A,\Omega^1_A)} : \Gamma(A,\Omega^1_A)\hookrightarrow \mb(\C)$ has the form $E : \C^g \hookrightarrow (\C^*)^{2g}$ sending $\vec{z}:=(z_1,\ldots,z_g) \mapsto (e^{z_1},\ldots,e^{z_g},e^{l_1(\vec{z})},\ldots,e^{l_g(\vec{z})})$ such that the $g \times g$ matrix of coefficients of the linear forms $l_1,\ldots,l_g$, lies in the Siegel upper half space. Namely, the matrix of coefficients is symmetric, and its imaginary part is positive-definite. Now the discreteness of $E(\C^g)\cap (\R^*)^{2g}$, boils down to the fact that the linear forms $z_1,\ldots,z_g, l_1(\vec{z}),\ldots,l_g(\vec{z})$ cannot all simultaneously be real unless $\vec{z} = 0$.  

In the course of the proof of \autoref{bi-algebraicity-prop}, we shall need the following Lemma that we prove first.

\begin{lemma}\label{torsion-lemma}
    Suppose $P$ is an algebraic group over $\C$ that is an extension of a vector group $\Omega$ over $\C$ and a complex abelian variety $B$. Thus we have an exact sequence:  
    $$0 \rightarrow \Omega \rightarrow P \xrightarrow{\pi} B \rightarrow 0.$$ Then $\pi$ induces an isomorphism on torsion subgroups $\pi : P(\C)_{\mathrm{tors}} \xrightarrow{\sim} B(\C)_{\mathrm{tors}}$. 
\end{lemma}
\begin{proof}
    $\pi : P(\C)_{\mathrm{tors}} \rightarrow B(\C)_{\mathrm{tors}}$ is indeed injective since $\Omega(\C)$ being a vector group is torsion-free. Let $b \in B(\C)[N]$ be an $N$-torsion point. Pick an element $p \in P(\C)$ such that $\pi(p) = b$. Then, $w := N\cdot p \in \Omega(\C)$. The group $\Omega(\C)$ being a vector group is uniquely divisible, and thus there is a unique $w' \in \Omega$ such that $N \cdot w' = w.$ Now, $\pi(p-w') = \pi(p) = b$ and $N\cdot (p-w') = N\cdot p -N\cdot w' = 0$. Thus, $(p-w') \in P(\C)_{\mathrm{tors}}$ and $\pi : P(\C)_{\mathrm{tors}} \rightarrow B(\C)_{\mathrm{tors}}$ is surjective, completing the proof of the Lemma.  
\end{proof}

\begin{proof}[Proof of \autoref{bi-algebraicity-prop}]
We shall induct on $g$. Let $S \subseteq T(\C)$ be an irreducible bi-algebraic subvariety and set $S' := \rho(S)\subseteq P(\C)$ as in the statement of the Proposition. By translating  $S$ and $S'$ we may assume that $S$ and $S'$ pass through the identity elements. Let $G_T$ (resp. $G_P$) denote the stabilizer of $S$ (resp. $S'$) inside $T$ (resp. $P$). We note that $G_T$ (resp. $G_P$) is a Zariski-closed algebraic subgroup of $T$ (resp. $P$). Evidently, $\rho$ induces an analytic isomorphism $\rho : G_T^\an \xrightarrow{\sim} G_P^\an$.

\textbf{\underline{Case 1}: Stabilizers are positive dimensional:}

Suppose that $G_P$ (and hence $G_T$) is positive dimensional. Let $G_P^o$ and $G_T^o$ denote the identity connected components of $G_P$ and $G_T$ respectively. $G_P^o$ (resp. $G_T^o$) is a finite-index algebraic subgroup of $G_P$ (resp. $G_T$). Then $\rho$ induces an analytic isomorphism on the quotients: $\overline{\rho} : (T/G_T^o)^\an \xrightarrow{\sim} (P/G_P^o)^\an$ and furthermore, $S$ and $S'$ are pullbacks of algebraic subvarieties of $T/G_T^o$ and $P/G_P^o$ respectively, that correspond to one another under $\overline{\rho}$. 

We would now like to apply the inductive hypothesis to the quotient $\overline{\rho}$. We now proceed to check that $\overline{\rho}$ indeed satisfies all the required hypothesis in the statement of the Proposition. 
Note that $T/G_T^o$ is again an algebraic affine torus over $\C$. On the other hand, $P/G_P^o$ sits in the following exact sequence: $0\rightarrow \Omega/(\Omega \cap G_P^o) \rightarrow P/G_P^o \xrightarrow{\pi} B/\pi(G_P^o)\rightarrow 0$ and since a quotient by a closed algebraic subgroup of the vector group $\Omega$ is still a vector group, $P/G_P^o$ is an extension of the abelian variety $B/\pi(G_P^o)$ by a vector group. Furthermore, $\overline{\Omega} := \Omega/(\Omega \cap G_P^o)$ and $B/\pi(G_P^o)$ have the same dimension, as can be seen, for instance, by counting torsion points. Namely, if $d:=\dim(T/G_T^o) = \dim(P/G_P^o)$  then for $N \geq 2$ we have $N^{d} = |(T/G_T^o) [N]| = |(P/G_P^o)[N]| = |B/\pi(G_P^o)[N]| = N^{2\cdot \dim(B/\pi(G_P^o))},$ where the second to last equality follows from \autoref{torsion-lemma}. Thus $d= 2 \dim(B/\pi(G_P^o))$. Let $g' := \dim(B/\pi(G_P^o)).$
We also claim that  $\overline{\rho}^{-1}(\overline{\Omega})\cap (T/G_T^o)(\R)$ is discrete in $(T/G_T^o)(\C).$ 
The claim that $\overline{\rho}^{-1}(\overline{\Omega})\cap (T/G_T)(\R)$ is discrete is equivalent to the claim that the intersections of the (real) Lie algebras $\Lie(\overline{\Omega}) \cap \Lie(\overline{\rho}((T/G_T^o)(\R)))$ inside $\Lie((P/G_P^o)(\C))$is $\{0\}$, which by the above dimension count is in turn equivalent to showing that $\Lie(\overline{\Omega}) + \Lie(\overline{\rho}((T/G_T^o)(\R))) = \Lie((P/G_P^o)(\C))$. 
Since we know that $\rho(T(\R))\cap \Omega$ is discrete in $P(\C)$, by the same reasoning we know that, $\Lie(\Omega)\oplus \Lie(\rho(T(\R))) = \Lie(P(\C)).$ 
It is easy to check from the definition of $T(\R)$ (see \autoref{def-of-T(R)}) that $(T/G_T^o)(\R) \cong T(\R)/G_T^o(\R).$ In particular, $\Lie(\rho(T(\R)))$ surjects onto $\Lie(\overline{\rho}((T/G_T^o)(\R)))$, under the natural surjection $\Lie(P(\C)) \twoheadrightarrow \Lie((P/G_P^o)(\C))$ and similarly we have the surjection $\Lie(\Omega)\twoheadrightarrow \Lie(\overline{\Omega})$. This therefore proves the required claim that $\Lie(\overline{\Omega}) + \Lie(\overline{\rho}((T/G_T^o)(\R))) = \Lie((P/G_P^o)(\C))$ and thus that $\overline{\rho}^{-1}(\overline{\Omega})\cap (T/G_T)(\R)$ is discrete.

Thus we may apply the inductive hypothesis to $\overline{\rho} : (T/G_T)^\an \xrightarrow{\sim} (P/G_P)^\an$ and the images of $S$ and $S'$ in $T/G_T$ and $P/G_P$ to conclude that $S/G_T$ and thus also $S$ is a translate of a subtorus. This concludes the case when the stabilizers $G_T$ and $G_P$ are positive dimensional subgroups.

\textbf{\underline{Case 2}: Stabilizers are finite:}

We now assume that the stabilizers $G_T$ and $G_P$ of $S$ and $S'$ are finite. We note that there are unique $\R^{\an,\exp}$-definable structures on $T(\C)$ and $P(\C)$ extending their algebraic structures. Fix an algebraic isomorphism $T(\C) \approx (\C^*)^{2g}$, and consider the universal covering map $u : \C^{2g} \rightarrow (\C^*)^{2g}$ that sends $(z_1,\ldots, z_{2g})\mapsto (\exp(2\pi i z_1),\ldots, \exp(2\pi i z_{2g}))$. Set $F := \{(z_1,\ldots,z_{2g}) \in \C^{2g} : \Re(z_i) \in [0,1)\}$. We then note that $u|_{F} : F \rightarrow T(\C)$ is an $\R^{\an,\exp}$-definable bijection. Consider $\rho \circ u : \C^{2g} \rightarrow P(\C),$ which is also a universal cover for $P(\C)$. The inclusion of the vector group $\iota : \Omega \hookrightarrow P(\C)$, lifts to a \emph{linear} injective map $\tilde{\iota} : \Omega \hookrightarrow \C^{2g}$. Choose a complex uniformization $p : \C^g \rightarrow B(\C)$; thus the kernel of $p$ is a rank $2g$ lattice, say $L$, in $\C^g$. We may lift $\pi : P(\C) \rightarrow B(\C)$ to a linear surjection $\tilde{\pi} : \C^{2g} \rightarrow \C^g$ thus obtaining the following commutative diagram with exact rows:

$$\begin{CD}
    0 @>>> \Omega @>\tilde{\iota}>> \C^{2g} @>\tilde{\pi} >> \C^g @>>> 0\\
      @.          @|    @VV\rho \circ u V              @VVp V    @. \\
    0 @>>> \Omega @>\iota>> P(\C) @>\pi >> B(\C) @>>> 0  
\end{CD}$$

Let $R \subseteq \C^g$ be a fundamental parallelopiped for $L$. Then $G := \tilde{\pi}^{-1}({R})$ is a fundamental domain for the action of $\Z^{2g}$ on $\C^{2g}$ and furthermore, $\rho \circ u \vert_G : G \rightarrow P(\C)$ is an $\R^{\an,\exp}$-definable bijection. For notational simplicity, let us call $v : G \rightarrow P(\C)$ this bijection $\rho \circ u \vert_G.$ We have a (\emph{non-definable}) bijection $\phi : F \rightarrow G$ defined by $\phi := v^{-1}\circ \rho \circ u\vert_F$. 

\begin{lemma}\label{bounded-lemma}
    $F \cap G$ is a bounded subset of $\C^{2g}.$ 
\end{lemma}
\begin{proof}
    $F \cap G$ is evidently a semi-linear subset of $\C^{2g} = \R^{4g}$. Suppose for the sake of contradiction that $F \cap G$ is unbounded. Then $F\cap G$ must contain an infinite real ray. More precisely, there exists an $a, b \in \C^{2g}$ such that for all $t >0$, $a + t b \in F \cap G$ and $b \neq 0$. Set $l := \{a+tb | t \in \R_{>0}\}$. Suppose $a = (a_1,\ldots,a_{2g}), b = (b_1,\ldots,b_{2g})$, then since for every $j = 1,\ldots, 2g$ and for all $t > 0$, $\Re(a_j+tb_j) \in [0,1)$, we must have that for each $j$, $b_j$ is purely imaginary.  Thus there exists an element $c = (c_1,\ldots,c_{2g}),$ so that $c \in \R^{2g}\setminus\{0\}$ and $b= i\cdot c.$ Since $l \subseteq G$, $\tilde{\pi}(l) \subseteq R$ which is a bounded subset of $\C^g$. Thus, $\tilde{\pi}\vert_l$ cannot be injective on $l$. Therefore, there exist two points $(a+t_1 b), (a+t_2 b) \in l$, with $t_1 \neq t_2$ such that $\tilde{\pi}(a+t_1b) = \tilde{\pi}(a+t_2b)$. Thus, the difference of these two points, $(t_1-t_2)b \in \tilde{\iota}(\Omega).$ Since $\tilde{\iota}(\Omega)$ is a complex vector subspace of $\C^{2g}$, this implies $(t_1-t_2)c \in \Omega$. In particular, $\tilde{\iota}(\Omega) \cap (\R^{2g}\setminus \{0\}) \neq \varnothing$ and again since $\tilde{\iota}(\Omega)$ is a linear subspace $\tilde{\iota}(\Omega)$ contains a real line in $\R^{2g}$. But this contradicts the assumption that $\Omega \cap \rho(T(\R))$ is discrete in $P(\C)$.
\end{proof}

Let $U' := v^{-1}(S')$ and $U := u|_F^{-1}(S)$. Thus, $U'$ and $U$ are both definable in $\C^{2g}$ and $\phi(U) = U'$. 
Now consider the \emph{definable} set:
\[ \Gamma := \{ (a,a') \in U \times U' | \exists \epsilon > 0, \exists \vec{x} \in \R^{2g}, \left(\vec{x} + (B(a,\epsilon)\cap U) = B(a',\epsilon)\cap U'\right) \land (\vec{x}+a = a') \}.\] 

Note that $\Gamma$ contains the graph of the bijection $\phi\vert_U :U \rightarrow U'.$ 

We claim that there is a finite subset $\Delta \subseteq \R^{2g}$ such that for every $a \in U,$ and elements $a'_1, a'_2 \in U'$ such that $(a,a'_1) \in \Gamma$ and  $(a,a'_2) \in \Gamma$, we have that $(a'_1-a'_2) \in \Delta$. Indeed, for such two elements $a'_1, a'_2 \in U'$ as above, $(U + (a'_2 - a'_1))\cap U$ contains a relatively open subset of $U$. Thus $S\cdot u(a'_2-a'_1) \cap S$ contains a relatively open subset of $S$. Since $S$ is an irreducible algebraic variety, this implies that $S\cdot u(a'_2-a'_1) = S$ and thus $u(a'_2-a'_1)$ belongs to the finite group $G_T$. Since $a'_2-a'_1 \in (-1,1)^{2g}$, there are only finitely many choices for $a'_2-a'_1$ which proves the claim. 

Consider the definable map $f : \Gamma \rightarrow \R^{2g}$ that sends $(a,a') \mapsto (a'-a)$. Since $\Gamma$ contains the graph of $\phi : U \rightarrow U'$, by the previous paragraph, we must have that $f(\Gamma) \subseteq (\Z^{2g}+\Delta)$ which is a discrete subset of $\R^{2g}.$ But $f(\Gamma)$ is definable in $\R^{\an, \exp}$ which is o-minimal, and thus $f(\Gamma)$ must have only finitely many connected components. This is only possible when $f(\Gamma)$ is a finite set. In particular, there are only finitely many $\vec{n}\in \Z^{2g}$ such that $U \cap (G+\vec{n}) \neq \varnothing.$ In other words, $U \subseteq \bigcup_{\vec{n}} \left(F\cap(G+\vec{n})\right)$ where the union is over only \emph{finitely} many $\vec{n} \in \Z^{2g}.$
By \autoref{bounded-lemma}, this implies that $U$ is a bounded subset of $\C^{2g}.$  
Thus, the closed subset $S \subseteq (\C^*)^{2g}$ is also bounded and therefore compact. On the other hand, $S$ being a closed algebraic subvariety of the affine variety $(\C^*)^{2g}$, must itself be affine. $S$ now being an \emph{irreducible}, affine, compact complex algebraic variety must be reduced to a single point. 

\end{proof}

\section{A \emph{p}-adic version for Mumford curves}\label{p-adic-proof}

In this section, we shall prove \autoref{p-adic-theorem}. We start by proving two Lemmas that we shall need later in the proof.
\begin{remark}\label{etale-analytic-implies-etale-algebraic}
We first recall that for a map of affinoid spaces $f : Y=\Sp(B) \rightarrow X=\Sp(A)$ over $\C_p$, we say that $f$ is \'etale if for every $y \in Y$ the induced map on local rings $\mathcal{O}_{X,f(y)} \rightarrow \mathcal{O}_{Y,y}$ is flat and unramified in the sense of \cite[p. 240]{fresnel2004rigid}. This in fact implies that $A \rightarrow B$ is flat.  
This follows from the following three facts: First, to show that $A \rightarrow B$ is flat it is enough to localize at every maximal ideal of $B$ (more precisely \cite[Theorem 7.1]{matsumura1989commutative}). Secondly, for local homomorphisms of Noetherian local rings, flatness can be checked on max-adic completions \cite[Theorem 22.4, p. 176]{matsumura1989commutative}. Thirdly, for a maximal ideal $\mathfrak{m}$ of the affinoid algebra $A$, the max-adic completion of $A_\mathfrak{m}$ is isomorphic to the max-adic completion of the Noetherian local ring $\mathcal{O}_{\Sp(A),\mathfrak{m}}$ (\cite[\S 4.1, Proposition 2]{bosch2014lectures}).
We now make the following elementary observation that we shall need. Let $f : \Sp(B) \rightarrow \Sp(A)$ be a \emph{finite} \'etale map of affinoids with $A, B$ integral domains, and such that $f$ is a set-theoretic bijection. Then $f$ is an isomorphism. Indeed, from the above we see that $A\rightarrow B$ is a finite, flat ring map and thus $B$ is a finite, locally-free $A$-module and the fact that $f$ is a bijection implies that $\Spec(B)\rightarrow \Spec(A)$ has generic fiber cardinality $1$ and thus is an isomorphism.  

\end{remark}

\begin{lemma}\label{analytic-subgroups-of-cp^n}
Let $B = \Sp(\C_p\{t\})$ be the rigid-analytic affinoid closed unit disk, that is $B = \{t \in \C_p : |t| \leq 1\}$. We view $B$ as a group object in the category of rigid analytic spaces over $\C_p$. Suppose $H \hookrightarrow B^n$ is a closed rigid-analytic subgroup of $B^n$. Then $H$ is linear. More precisely, there is a $\C_p$-linear subspace $L$ of $\C_p^n$ such that $H = L \cap B^n$. 
\end{lemma}
\begin{proof}
We prove this by induction on $n$. Let $H$ be as in the statement of the Lemma. Suppose $n=1$, and that $H \neq \{0\}$. For $a \in H\setminus\{0\}$, the set $\Z\cdot a$ is contained in $H$ and has $0$ as an accumulation point. This implies that $H$ cannot be zero-dimensional and thus $H = B$. This completes the base case.

Suppose now that $n \geq 2$ and that $H \neq \{0\}$. We may apply a linear automorphism of $B^n$ (i.e. an element of $\gl_n(\OCp)$) and replace $H$ by its image under such an automorphism to assume that $(a,0,\ldots,0) \in H$ for some $a \in \OCp$.

Now $H\cap (B \times \{\underline{0}\})$ is a closed rigid-analytic subgroup of  $B \times \{\underline{0}\}$ that is not trivial, and thus from the base case we must have $H\cap (B \times \{\underline{0}\}) = (B \times \{\underline{0}\})$. By the induction hypothesis, $H \cap (\{0\}\times B^{n-1})$ is a linear in $B^{n-1}$. Now the Lemma follows since $H =B \times \left(H \cap (\{0\}\times B^{n-1})\right)$
\end{proof}

\begin{lemma}\label{analytic-subgroups-of-gm^n}
Suppose $H \hookrightarrow (\Gm^n)^\an_{\C_p}$ is a torsion-free, closed rigid-analytic subgroup of $(\Gm^n)^\an_{\C_p}$. Then $H(\C_p)$ is discrete in $(\C_p^\times)^n$. Equivalently, the intersection of $H(\C_p)$ with every affinoid subdomain of $(\Gm^n)^\an_{\C_p}$ is finite. 
\end{lemma}
\begin{proof}
Let $D := \{(z_1,\ldots,z_n) \in (\C_p^\times)^n \mid \forall j\in\{1,\ldots,n\}, |z_j-1| < 1\}$ be the open unit polydisk around $1$. Note that $D$ is an admissible open in $(\Gm^n)^\an_{\C_p}$ and the logarithm map $\log : D \rightarrow \C_p^n$ that sends $(z_1,\ldots,z_n) \mapsto (\log(z_1),\ldots,\log(z_n))$ is an \'etale, Galois covering of rigid-analytic group varieties with kernel being the $n$-fold product of the group of $p$-power roots of unity. 

Let $E$ be a sufficiently small closed affinoid disk contained in $D$ and around $1$, such that the restriction $\log\vert_E : E \rightarrow \C_p^n$ is a rigid-analytic isomorphism onto the image $E' := \log(E)$ which is a closed affinoid disk in $\C_p^n$ containing $0$. We note that $E$ is a multiplicative subgroup of $D$. Thus, $H \cap E$ is a closed rigid-analytic multiplicative subgroup of the disk $E$, and so $\log(H\cap E)$ is a closed rigid-analytic additive subgroup of the affinoid disk $E'$ in $\C_p^n$ around $0$. By \autoref{analytic-subgroups-of-cp^n}, we get that there is a linear subspace $L' \subseteq \C_p^n$ such that $\log(H \cap E) = L'\cap E'$. Let $L := \log^{-1}(L')$ so that $L$ is a closed rigid-analytic multiplicative subgroup of $D$ and such that $H \cap E \subseteq L \cap E$. Let $H_0$ be the identity rigid-analytic connected component of $H \cap D$ and let $L_0$ similarly be the identity component of $L$. Since $H_0$ and $L_0$ are both analytic group varieties, they must also be analytically irreducible. Since $H \cap E = H_0 \cap E \subseteq L_0\cap E$, the irreducibility of $H_0$ implies that $H_0 \subseteq L_0$. On the other hand, $L_0$ and $H_0$ have the same dimension and $L_0$ being irreducible implies that $H_0 = L_0$. 

We shall now prove below that $L_0$ cannot be torsion-free unless it is the trivial subgroup of $D$. Indeed, since the cover $\log : D \rightarrow \C_p^n$ is Galois, we see that all the irreducible components of $L = \log^{-1}(L')$ are in fact connected and are permuted by the Galois group of the cover. Thus, the map $L_0 \rightarrow L'$ is a \emph{surjective} rigid-analytic homomorphism. 
Furthermore, the fact that $L_0$ is torsion-free implies that the map $L_0 \rightarrow L'$ is also injective, and thus a bijection. 
The bijection $L_0 \rightarrow L'$ is the composition of the open and closed immersion $L_0 \hookrightarrow L$ with the \'etale, Galois surjection $L \rightarrow L'$. 
Consider an admissible cover of $L'$ by \emph{connected} rigid-analytic affinoid subdomains $L' = \cup_i \Sp(A_i)$ such that the preimage of each $\Sp(A_i)$ in $L$ is a disjoint union of countably many \emph{connected} affinoids $\coprod_{n \in I_i} \Sp(B_{n,i})$ with each $B_{n,i}$ being finite \'etale over $A_i$. We note that the index sets $I_i$ are at most countable.
The existence of such a cover follows from the fact that there exists such a cover for $\log : D \rightarrow \C_p^n$ (see for instance \cite{dejong1995etale}).

Since each $\Sp(B_{n,i})$ is connected, $\Sp(B_{n,i})\cap L_0 \neq \varnothing$ if and only if $\Sp(B_{n,i})\subseteq L_0$. For each $i$, let $J_i \subseteq I_i$ be the subset such that $n \in J_i$ iff $\Sp(B_{n,i}) \subseteq L_0$.
Then we note that \[\Sp(A_i) = \coprod_{n \in J_i} \left(\text{image of } (\Sp(B_{n,i}) \rightarrow \Sp(A_i))\right).\]
The union on the right hand side above is an at most countable disjoint union of closed analytic subsets of $\Sp(A_i)$ which by the Baire category theorem is only possible if the union on the right hand side is over a singleton set $J_i$. Thus, for each $A_i$, there is a \emph{unique} $n(i) \in I_i$ such that $\Sp(B_{n(i),i}) \cap L_0 \neq \varnothing$ and in this case $\Sp(B_{n(i),i}) \subseteq L_0$ and thus $\Sp(B_{n(i),i})\rightarrow \Sp(A_i)$ being a finite, \'etale bijection must be an isomorphism (\autoref{etale-analytic-implies-etale-algebraic}). The union $\cup_{i} \Sp(B_{n(i),i})$ is an admissible affinoid covering of $L_0$ and therefore this implies that the map $L_0 \rightarrow L'$ is an isomorphism of rigid-analytic groups. 

Let $\phi : L' \rightarrow L_0$ denote the inverse of this map. Then $L' \xrightarrow{\phi} L_0 \rightarrow H_0 \hookrightarrow (\Gmcp^{n})^\an$ is a rigid analytic map from the linear space $L'$ to $(\Gmcp^{n})^\an$. But a Newton polygon argument implies that there are no non-constant global non-vanishing rigid analytic functions on a positive-dimensional linear space $L'$ (see \cite[Lemma 4.1.1-(2)]{conrad1999irreducible}).  
Thus, $L'$, $L_0$ and $H_0$ are all singleton sets. 
This concludes the proof that $H$ is discrete.
 
\end{proof}

We now turn to the proof of \autoref{p-adic-theorem}. We recall the setup and notation from \autoref{introduction}. We start with a Schottky group $\Gamma \subset \pgl_2(\C_p)$ which is a free group of rank $g \geq 1$. Denote by $\mathcal{L}_\Gamma \subseteq \mathbb{P}^1(\C_p)$ the limit set of the Schottky group $\Gamma$ and let $\Omega_\Gamma \subseteq \mathbb{P}^1(\C_p)$ denote the complement of $\mathcal{L}_\Gamma$. The rigid analytic quotient, $\Omega_\Gamma/\Gamma$ is a smooth, proper $\C_p$-analytic space and thus uniquely the analytic space attached to a genus $g$ smooth, projective curve over $\C_p$ that we denote by $X$. Thus $X^\an \cong \Omega_\Gamma/\Gamma$. 

As in the Introduction - \autoref{introduction}, we let $\mb(X) = \Hom(\pi_1^\topo(X^\an), \C_p^\times) = \Hom(\Gamma, \C_p^\times)$ be the moduli space of topological rank one $\C_p$-local systems on $X^\an$ and $\mdr(X)$ be the moduli space of algebraic line bundles equipped with connection on $X$. Both $\mb(X)$ and $\mdr(X)$ have a natural structure as algebraic groups over $\C_p$. The Riemann--Hilbert functor discussed in \autoref{introduction}, provides an injective rigid-analytic group homomorphism $\mb(X)^\an \hookrightarrow  \mdr(X)^\an$ (\cite[p.\,143]{andre1998p}). We note that $\mb(X) = \Spec(\C_p[\Gamma])$ is an affine algebraic torus of dimension $g$ while $\mdr(X)$ is the universal vector extension of the Jacobian sitting in the exact sequence: \[0\rightarrow H^0(X,\Omega^1_X)\rightarrow \mdr(X)\rightarrow \Pic^0(X)\rightarrow 0.\]
\begin{remark}
In \cite[p.\,143]{andre1998p}, Andr\'e asks the question of whether the injective analytic map $\mb(X)^\an \hookrightarrow \mdr(X)^\an$ is a closed immersion or not. However, we argue that at least for Mumford curves and in the rank one case, it is not a closed immersion for the following group-theoretic reason. Suppose that $\mb(X)^\an \hookrightarrow \mdr(X)^\an$ was indeed a closed immersion. Then $H^0(X,\Omega^1_X)\cap \mb(X) \hookrightarrow H^0(X,\Omega^1_X)$ is a closed analytic subgroup and so is the inclusion $H^0(X,\Omega^1_X)\cap \mb(X) \hookrightarrow \mb(X)^\an$. By \autoref{analytic-subgroups-of-gm^n}, we see that on the one hand $H^0(X,\Omega^1_X)\cap \mb(X)$ being torsion-free and a closed analytic subgroup of $\mb$ is discrete in $\mb$ and thus at most countable. On the other hand, $\mb(X)\cap H^0(X,\Omega^1_X)$ is closed-analytic in $H^0(X,\Omega^1_X)$ and so by \autoref{analytic-subgroups-of-cp^n} must be a linear subspace and thus uncountable or singleton. This tells us then that $H^0(X,\Omega^1_X)\cap \mb(X)$ is trivial. Therefore, the composite map $\mb(X)^\an \rightarrow \Pic^0(X)^\an$ is an injective rigid-analytic group homomorphism. But this is evidently not possible.
\end{remark} 
We aim to characterise algebraic subvarieties of $\mdr(X)$ that intersect $\mb(X)^\an$ in an irreducible algebraic subvariety of $\mb(X)$. We shall prove the following slightly more general claim.

\begin{proposition}\label{p-adic-slightly-general-prop}
    Let $T$ be an affine algebraic torus over $\C_p$ of dimension $g$. Let $P$ be an algebraic group over $\C_p$ that is the extension of an abelian variety $B$ over $\C_p$ by a vector group $V$. Thus, we have the following exact sequence of commutative algebraic groups over $\C_p$: \[0 \rightarrow V \xrightarrow{i} P(\C_p) \xrightarrow{\pi} B(\C_p) \rightarrow 0.\] Suppose $\rho : T^\an \hookrightarrow P^\an$ is an injective rigid-analytic group homomorphism such that $\pi \circ \rho : T^\an \rightarrow B^\an$ is surjective on $\C_p$-points. Suppose $S' \subseteq P(\C_p)$ is an algebraic subvariety of $P$ such that $S:=\rho^{-1}(S')$ is an irreducible algebraic subvariety of $T$. Then $S$ must be a translate of a subtorus of $T$. 
\end{proposition}

\begin{remark}
We remark that \autoref{p-adic-slightly-general-prop} implies \autoref{p-adic-theorem}. To see this, it suffices to note that for a Mumford curve $X$, the composite map $\mb(X)^\an \hookrightarrow \mdr(X)^\an \rightarrow \Pic^0(X)^\an$ is surjective on $\C_p$ points (see \cite[Proposition on p. 184]{gerritzen1980schottky}) and thus the hypotheses to be able to apply \autoref{p-adic-slightly-general-prop} hold. 
\end{remark} 

\begin{proof}[Proof of \autoref{p-adic-slightly-general-prop}]

Let $S'$ and $S$ be as in the statement of \autoref{p-adic-slightly-general-prop}. 
We may replace $S'$ by $\overline{\rho(S)}^\mathrm{Zar}$, the Zariski closure of the image of $S$ in $P(\C_p)$, and assume from now on that $S' = \overline{\rho(S)}^\mathrm{Zar}$.
Let $G_P := \{g \in P(\C_p) : g\cdot S' = S'\}$ and $G_T := \{g \in T(\C_p) : g\cdot S = S\}$ denote the stabilizers of $S'$ and $S$ respectively. $G_P$ (resp. $G_T$) is an algebraic subgroup of $P(\C_p)$ (resp. $T(\C_p)$). 

We note that $\rho^{-1}(G_P) = G_T$ since $\rho^{-1}(S') = S$.
Set $P' := P/G_P$, and $T' := T/G_T$. The map $\rho$ descends to an injective rigid-analytic group homomorphism $\rho' : T'^\an \hookrightarrow P'^\an$. 
Let $V' := V/(V\cap G_P)$ and $B' := B/\pi(G_P)$ so that $V'$ is a vector group and $B'$ is an abelian variety over $\C_p$. 
We have the exact sequence of commutative algebraic groups over $\C_p$: \[0\rightarrow V' \xrightarrow{i'} P'(\C_p) \xrightarrow{\pi'} B'(\C_p)\rightarrow 0.\]
Evidently, $\pi' \circ \rho' : T'(\C_p) \rightarrow B'(\C_p)$ is still surjective.
$S'$ is the inverse image in $P(\C_p)$ of an algebraic subvariety (say $S'_0$) of $P'(\C_p)$ and similarly $S$ is the inverse image in $T(\C_p)$ of an irreducible algebraic subvariety (say $S_0$) of $T'(\C_p)$. It is easy to see that $(\rho')^{-1}(S'_0) = S_0$ and moreover that $\overline{\rho'(S_0)}^\mathrm{Zar} = S'_0$. 

We also see that the stabilizers of $S'_0$ in $P'(\C_p)$ and that of $S_0$ in $T'(\C_p)$ are both trivial. 
Furthermore, if we are able to prove that $S_0$ is a translate of a subtorus of $T'$, then $S$ being the pre-image of $S_0$ in $T$ shall also be a translate of a subtorus. We may thus replace $T, P, V, B, \rho, S, S'$ by $T', P', V', B', \rho', S_0, S'_0$.

\emph{We may therefore assume from this point onwards that the stabilizers $G_T$ and $G_P$ are both trivial}.

Consider the map $i \oplus \rho : V^\an \times T^\an \rightarrow P^\an$. 
For simplicity, let us call this map $\phi : V^\an \times T^\an \rightarrow P^\an$. 
Since $\pi \circ\rho : T(\C_p) \rightarrow B(\C_p)$ is surjective, we find that $\phi$ is also surjective. 
For notational convenience, denote the composite map $\pi\circ \rho : T(\C_p) \rightarrow B(\C_p)$ by $q : T(\C_p) \rightarrow B(\C_p)$ instead.
The kernel of $q : T(\C_p) \twoheadrightarrow B(\C_p)$ is $V \cap T(\C_p)$. Since $V$ is a vector group that is a closed analytic subgroup of $P$, the kernel of $q$, $\ker(q)$, is a \emph{torsion-free} closed analytic subgroup of $T^\an$. By \autoref{analytic-subgroups-of-gm^n}, $\ker(q)$ is therefore a discrete subgroup of $T^\an$. The quotient of $T$ by the discrete subgroup $\ker(q)$ is proper since $B$ is an abelian variety. This implies that $\ker(q)$ is a full-rank lattice in $T^\an$ (\cite[Proposition 2.7.3]{lutkebohmert2016rigid}). 
We let $F \subset T(\C_p)$ denote a connected bounded fundamental domain for the lattice $\ker(q)$ acting on $T(\C_p)$. 

We now impose on $P(\C_p)$ the unique rigid-subanalytic structure that extends its algebraic structure. 
The map $\phi\vert_{V \times F} : V \times F \rightarrow P(\C_p)$ is a rigid-subanalytic bijection (which follows from the fact that $F$ is bounded and hence that $q\vert_F : F \rightarrow B(\C_p)$ is a rigid-subanalytic bijection). Let $S'_1$ be the rigid-subanalytic set given by $S'_1 := \phi^{-1}(S')\cap (V \times F)$. For $v \in V$, we denote by $(S'_1)_v$ the rigid-subanalytic subset of $F$ defined by $\{f \in F : (v,f) \in S'_1\}.$ 

Define \[D := \{(v,f,s) \in V \times F \times S : ((v,f) \in S'_1)\land \left(\exists \epsilon > 0 | f\cdot S\cap B(s,\epsilon) = s\cdot (S'_1)_v \cap B(f, \epsilon)\right)\}.\] Note that $D$ is rigid-subanalytic and contains $\{(v,f,s) \in V \times F \times S : \phi(v,f) = \rho(s)\}$. Furthermore, for a point $(v,f,s)\in D$, $(s^{-1}f)\cdot S\cap B(s,\epsilon) \subseteq (S'_1)_v$ implies that $\rho(S\cap B(s,\epsilon)) \subseteq S'-i(v)-\rho(s^{-1}f)$ (where the $+$ here denotes to the group operation of $P(\C_p)$). Since $S$ is analytically irreducible, this implies that $\rho(S) \subseteq S'-i(v)-\rho(s^{-1}f)$ and thus that $S'=\overline{\rho(S)}^\mathrm{Zar}\subseteq S'-i(v)-\rho(s^{-1}f).$ In other words, $i(v)+\rho(s^{-1}f)$ is in the stabilizer $G_P$ of $S'$ in $P(\C_p)$.  
However, we remind ourselves that the stabilizer is $G_P$ is trivial. Therefore, $i(v) + \rho(s^{-1}f) = 0 $ in $P(\C_p)$ which implies that $s^{-1}f \in \rho^{-1}(i(V)) = \ker(q)$ which we recall is a discrete lattice in $T(\C_p).$

Thus we have proved that the map $g : D \rightarrow T(\C_p)$ that sends $(v,f,s) \in D$ to $s^{-1}f$ has image that is contained in the discrete lattice $\ker(q)$. On the other hand, this map is evidently rigid-subanalytic. 
This implies that the image $g(D)$ is \emph{finite}. 
For every $s_0 \in S$, there is some $(v_0,f_0) \in V\times F$ such that $\phi(v_0,f_0) = \rho(s_0)$ (so that $(v_0,f_0,s_0) \in D$), and thus $s_0^{-1}f_0$ belongs to $g(D)$ and so has only finitely many choices. This implies that $S$ is contained in only finitely many translates of the bounded fundamental domain $F$, and hence that $S \subseteq T(\C_p)$ is bounded. An irreducible algebraic subvariety $S$ of $T$ is bounded in $T$ if and only if $S=\{1\}$ (this follows for instance using Noether normalization to produce a finite surjective morphism from $S$ to affine space of the same dimension. This produces an example of a global regular function on $S$ that is unbounded if $S$ is positive-dimensional. On the other hand if $S$ is bounded in $T(\C_p)$, every regular function on $S$ would also be bounded). 

This completes the proof of the Proposition.
\end{proof}

\bibliographystyle{alpha}
\bibliography{main}

\textsc{California Institute of Technology, Pasadena, California}

\textit{E-mail address}: \texttt{abhishek@caltech.edu}
\end{document}